\documentclass[11pt]{article}

\usepackage{amssymb, amsthm}
\usepackage{amsmath,mathrsfs,latexsym,amsfonts}

\newtheorem{theorem}{Theorem}
\newtheorem{lemma}{Lemma}
\newtheorem{proposition}{Proposition}
\newtheorem{corollary}{Corollary}
\newcommand\proofof[1]{{\medskip\noindent\em Proof of #1.\quad}}
\def\eproof{{\mbox{}\hfill\qed}\medskip}

\def\R{{\mathbb{R}}}
\def\N{{\mathbb{N}}}
\def\Prob{\mathop{\mathsf{Prob}}}
\def\bE{\mathop{\mathbb{E}}}
\def\Oh{{{\cal O}}}

\def\cT{{{\cal T}}}

\def\DD{{{\cal D}}}
\def\cS{{\cal S}}
\def\cdet{{\boldsymbol c}^{\mathsf{det}}}
\def\ci{{\boldsymbol c}^{\dagger}}
\def\ckl{{\boldsymbol c}^{\dagger}_{kl}}
\def\ck{{\boldsymbol c}_{k}}
\def\c{{\boldsymbol c}}
\def\MM{{\mathscr M}}

\def\Id{\mathsf{Id}}
\def\RelError{\mathsf{RelError}}
\def\LoP{\mathsf{LoP}}
\def\cond{\mathsf{cond}}
\def\emac{\varepsilon_{\mathsf{mach}}}

\def\BAMS{Bulletin of the Amer. Math. Soc.}

\begin{document}

\begin{title}
 {\LARGE \bf Smoothed analysis of componentwise condition numbers for 
sparse matrices}
\end{title}
\author{Dennis Cheung\\
United International College\\
Tang Jia Wan\\
Zhuhai, Guandong Province\\
P.R. of CHINA\\
e-mail:
{\tt dennisc@uic.edu.hk}
\and
Felipe Cucker
\thanks{Partially supported by 
GRF grant CityU 100808}
\\
Department of Mathematics\\
City University of Hong Kong\\
83 Tat Chee Avenue, Kowloon\\
HONG KONG\\
e-mail:
{\tt macucker@cityu.edu.hk}
}

\date{}

\maketitle

\begin{quote}
{\footnotesize {\bf Abstract.\quad}
We perform a smoothed analysis of the componentwise condition numbers 
for determinant computation, matrix inversion, and linear equations solving
for sparse $n\times n$ matrices.  
The bounds we obtain for the expectations of the logarithm of these  
condition numbers are, in all three cases, of the order $\Oh(\log n)$. 
As a consequence, small bounds on the smoothed 
loss of accuracy for triangular linear systems follow.  
}
\end{quote}

\section{Introduction}

The most commonly used solver of linear systems of equations, 
Gaussian elimination, reduces the input system $Ax=c$ to a system 
$Lx=b$ with $L$ lower triangular (and same solution $x$). 
Then, it solves the latter by forward substitution. As a consequence, 
triangular systems of equations are routinely solved by computer.

Almost on every occasion, the accuracy of the computed solution 
is very high. Yet, the reasons for this accuracy have been dodging 
researchers for quite a while. In the early 1960s 
J.H.~Wilkinson noted that ``In practice one almost invariably finds 
that if $L$ is ill-conditioned, so that $\|L\|\|L^{-1}\|\gg1$, then the 
computed solution of $Lx=b$ (or the computed inverse) is far more accurate 
than [what forward stability analysis] would suggest'' ~\cite[p.~105]{Wilkinson63}. 
To make things worse, ill-conditioned matrices $L$ in the sense above, 
appeared to be ubiquitous. This was explained by 
by D.~Viswanath and N.~Trefethen in~\cite{ViTr98}. 
Indeed, if $L_n$ denotes a 
random triangular $n\times n$ matrix (whose entries are independent 
standard Gaussian random variables) and 
$\kappa_n=\|L_n\|\|L_n^{-1}\|$ is its condition number (which is a 
positive random variable) then, the main result in~\cite{ViTr98} 
shows that
$$
   \sqrt[n]{\kappa_n}\to 2 \quad\mbox{almost surely}
$$ 
as $n\to\infty$. A straightforward consequence of this result is that 
the expected value of $\log\kappa_n$ satisfies 
$\bE(\log\kappa_n)=\Omega(n)$. 

Putting all the above together we can describe the situation as follows:
\begin{quote}
{\sl Triangular systems of equations are generally solved to high accuracy 
in spite of being, in general, ill-conditioned.} 
\end{quote}

In 1989 N.~Higham~\cite{Higham:89} pointed out that the backward error 
analysis given by Wilkinson for forward substitution 
yields (small) {\em  componentwise} bounds on the perturbated 
matrix. One can therefore deduce small forward error bounds 
for these solutions if the {\em componentwise condition number} 
$\c(L,b)$ of the 
pair $(L,b)$ ---instead of $\kappa(L)$--- is small. In a recent 
paper~\cite{ChC08} we showed that this is the case for random 
triangular matrices $L$. Here `random' means that the entries 
of $L$ are i.i.d.~standard random variables. This result provides
an explanation of the high accuracy achieved in general by 
forward substitution. 

In the last decade, however, the suitability of 
this {\em average analysis} to reflect performance of algorithmic practice 
was questioned.  The objection raised is that the probability 
distribution underlying these analyses ---usually, a centered 
isotropic Gaussian--- is chosen because of technical reasons and not 
because it models ``the real world.'' Because of this, it may well happen 
that the resulting estimates are too optimistic, just as worst-case analysis 
is often claimed to be too pessimistic. The proposed alternative, 
{\em smoothed analysis}, interpolates between worst-case and average 
analyses and typically studies, for a function $f:\R^p\to\R$, the 
quantity
$$
            \sup_{\bar{a}\in\R^p}\bE_{a\sim N(\bar{a},\sigma^2\Id)} f(a).
$$
Here $N(\bar{a},\sigma^2\Id)$ denotes the normal 
distribution centered at 
$\bar{a}$ and with covariance matrix $\sigma^2\Id$, 
where $\Id$ is the identity matrix. In case $f$ 
is homogeneous (i.e., $f(\lambda a)=f(a)$ for all $\lambda\neq 0$)
it is common to scale the covariance matrix and study 
$$
     \sup_{\bar{a}\in\R^p}\bE_{a\sim N(\bar{a},\sigma^2\|\bar{a}\|\Id)} f(a).
$$
or, equivalently, 
$$
            \sup_{\|\bar{a}\|=1}\bE_{a\sim N(\bar{a},\sigma^2\Id)} f(a).
$$
In this case, the interpolation 
mentioned above comes from the fact that when $\sigma=0$ the 
expression above reduces to the worst-case of $f$ and when 
$\sigma\to\infty$ one approaches the usual average analysis. 
We won't elaborate here on the virtues of smoothed analysis. 
The interested reader can find expositions of these virtues 
in~\cite{ST:02,ST:04,ST:06,ST:09} or~\cite[\S2.2.7]{Condition}. 
We will instead proceed to 
state the main results of this paper. 
For a matrix $A$ we define the max norm 
$$
  \|A\|_{\max}=\max_{ij}|a_{ij}|.
$$

\begin{theorem}\label{thm:LES} 
Let $\cT$ denote the set of $n\times n$ lower triangular matrices. 
Let $\bar{L}\in\cT$ and $\bar{b}\in\R^n$ be such that 
$\|\bar{L}\|_{\max}\leq 1$ and $\|\bar{b}\|_\infty\leq 1$. 
For $L\in\cT$ and  $b\in\R^n$ let $\c(L,b)$ denote the componentwise 
condition number, for the problem of linear equation solving, 
of the pair $(L,b)$. Then, for any real number $t>n(n+1)$ we have
\begin{equation*}
\Prob_{(L,b)\sim N_{\cT}((\bar{L},\bar{b}),\sigma^2\Id)}
\{\c(L,b)>t\} \,\leq\,\left(\frac{1+\sigma}{\sigma}\right)
\left(\frac{n^3(n+1)^2}{t-n(n+1)}\right)\sqrt\frac{2}{\pi}
\end{equation*}
and, for any $\beta>1$,
\begin{equation*}
\bE_{(L,b)\sim N_{\cT}((\bar{L},\bar{b}),\sigma^2\Id)}(\log_\beta(\c(L,b)))
\,\leq \,\log_\beta\left(\frac{1+\sigma}{\sigma}\right)+5\log_\beta(n)
+\frac{2.65}{\ln\beta}.
\end{equation*}
The subindex $\cT$ in $N_{\cT}((\bar{L},\bar{b}),\sigma^2\Id)$ 
is meant to denote that $L$ is triangular. That is, 
the only entries of $L$ which are drawn 
from the Gaussian $N((\bar{L},\bar{b}),\sigma^2\Id)$ are those 
in its lower part.
\end{theorem}

This theorem has immediate consequences for the accuracy 
of forward substitution. Recall (or look at the {\em Overture} chapter 
in~\cite{Condition} for a primer if you are not familiar with 
round-off analysis), a finite precision algorithm with 
{\em machine precision} $\emac$ rounds-off all the real numbers $z$ 
occuring in the execution to a rational (floating point) number 
$\tilde z$ satisfying 
$$
    \RelError(z):=\frac{|\tilde{z}-z|}{|z|}\leq \emac
$$ 
(we agree this equality to hold if $z=\tilde{z}=0$). This means that 
the approximation $\tilde{z}$ has $\log_{10}(\frac{1}{\emac})$ 
correct (significant) digits\footnote{All our discussion holds as well for 
bits, instead of digits. The modifications required are trivial.}. 

If we solve a system $Lx=b$ with a finite precision machine we 
obtain an approximation $\tilde{x}$ of the solution $x$. A (componentwise) 
extension of the notion above measures the relative error of this approximation 
by 
$$
  \RelError(x):=\max_{i\leq n}\RelError(x_i).
$$
Again, $\log_{10}(\RelError^{-1}(x))$ provides a lower bound on the 
number of correct digits for {\em all} the components of $x$ and hence 
the {\em loss of precision} in the computation of $x$ ---i.e., the initial 
precision of our data measured in number of correct digits minus the precision 
of the computed outcome measured in the same manner--- is  
$$
   \LoP(x):=\log_{10}(\emac^{-1})-\log_{10}(\RelError^{-1}(x)).
$$
Note that if $L$ is singular then $x$ is not well-defined or may not exist. 
In this case we take, by convention, $\LoP(x)=\infty$.  
The following result provides a smoothed analysis of this 
quantity for forward substitution with finite precision.

\begin{corollary}\label{cor:LES} 
Assume we solve systems $Lx=b$ using forward substitution. Then, 
for all $\bar{L}\in\cT$ and $\bar{b}\in\R^n$ with  
$\|\bar{L}\|_{\max}\leq 1$ and $\|\bar{b}\|_\infty\leq 1$ 
we have
\begin{equation*}
\bE(\LoP(x)) 
= \log_{10}\left(\frac{1+\sigma}{\sigma}\right)+5\log_{10} n 
+ \log_{10}(\log_2 n)+ 1.452+o(1).
\end{equation*}
Here $(L,b)\sim N_{\cT}((\bar{L},\bar{b}),\sigma^2\Id)$
and $o(1)$ is a quantity that tends to zero with $\emac$.
\end{corollary}

\section{Preliminaries}

\subsection{Componentwise condition numbers}

Condition numbers measure the worst-case magnification 
in the computed outcome of a small perturbation 
in the data. As originally introduced by Turing~\cite{Turing48},
or von Neumann and Goldstine~\cite{vNGo47}, 
they were {\em normwise} in the sense that data perturbation and  
outcome's error were measured using norms (in the space of data and outcomes 
respectively). In contrast, {\em componentwise} condition numbers measure 
both of them componentwise. 

For both data perturbation and output error, the error is measured 
in a relative manner. Because of this, the following
form of ``distance'' function (it is not a distance as is not symmetric) 
will be useful to define componentwise condition numbers.
For points $u,v \in \R^p$ we define 
$\frac{u}{v}=(w_1,\ldots,w_p)$ with 
$$
w_i=\left\{ \begin{array}{ll}
             u_i/v_i &\mbox{if $v_i\neq 0$}\\ 
             0 &\mbox{if $u_i=v_i=0$}\\ 
             \infty &\mbox{otherwise.}
             \end{array}\right.
$$
Then we define
$$
   d(u,v):=\left\|\frac{u-v}{v}\right\|_\infty.
$$
Note that, if $d(u,v)<\infty$, 
$$
  d(u,v):=\min\{\nu \geq 0\mid |u_i-v_i|\leq\nu|v_i|
              \mbox{ for $i=1,\ldots,p$}\}.
$$
For $\delta>0$ and $a\in\R^p$ we denote
$\cS(a,\delta)=\{x\in\R^p\mid d(x,a)\leq \delta\}$. 

Let $\DD\subseteq\R^p$, $F:\DD \rightarrow \R^q$ be a continuous
mapping, and $a \in \DD$ be such that $a_j\neq 0$ for $j=1,\ldots,q$. 
Then the {\em componentwise condition number} of $F$ at $a$ is
\begin{equation}\label{def:c}
  \c^F(a):=\lim_{\delta \rightarrow 0} 
    \sup_{x \in \cS(a,\delta) \atop x \neq a}
    \frac{d(F(x),F(a))}{d(x,a)}.
\end{equation}
It is not difficult to see that 
\begin{equation}\label{char:c}
   \c^F(a)=\max_{j\leq q}\c^{F_j}(a)
\end{equation}
where $\c^{F_j}(a)$ denotes the componentwise condition number 
of $a$ for the $j$th component $F_j$ of $F$. We will systematically 
use this form in the rest of this paper.

\subsection{Sparse matrices}

In all what follows, for $n\in\N$, 
we denote the set $\{1,\ldots, n\}$ by $[n]$.

We denote by $\MM$ the set of $n\times n$ real matrices and by 
$\Sigma$ its subset of singular matrices. 
Also, for a subset $S\subseteq[n]^2$ we denote 
$$
  \MM_S=\{A\in\MM \mid \mbox{ if 
  $(i,j)\not\in S$ then $a_{ij}=0$}\}. 
$$ 
Matrices in $\MM_S$ for some $S\neq [n]^2$ 
(i.e.~matrices with a fixed pattern of zeros) are said to be {\em sparse}. 
The set $S$ is said to be {\em admissible} if $\MM_S$ contains 
some invertible matrix. 

In the rest of this paper, for non-singular matrices $A,A'$, we denote 
their inverses by $\Gamma,\Gamma'$, respectively. Also, we denote 
by $A_{ij}$ the sub-matrix of $A$
obtained by removing from $A$ its $i$th row and its $j$th column.

The technical results below, Theorems~\ref{thm:det},~\ref{result2} 
and~\ref{result3}, 
are proved in the general context of sparse matrices. Besides triangular 
matrices, these results apply to other classes of sparse matrices such
as, for instance, tridiagonal matrices. 

\subsection{Smoothed analysis}\label{s2}

Let $\sigma>0$ be a fixed number, $S\subset [n]^2$ be admissible 
and $\bar{A}=(\bar{a}_{ij})\in\MM_S$. 
Extending the notation we used in the Introduction, we will 
write $A\sim N_S(\bar{A},\sigma^2\|\bar{A}\|_{\max}\Id)$ to 
denote that the entry $a_{ij}$ of $A$, with $(i,j)\in S$, is 
a random variable with distribution 
$N(\bar{a}_{ij},\sigma^2\|\bar{A}\|_{\max})$, whereas the entries 
$a_{ij}$ with $(i,j)\not\in S$ are zero. 

In this paper we will only be concerned, for 
a random sparse matrix $A$ as above, with the 
componentwise condition number 
of $A$ with respect to a few problems. All these condition 
numbers being, as functions, homogeneous of degree 0, 
we will sistematically consider, without loss of generality, 
the center $\bar{A}$ of the 
distribution to satisfy $\|\bar{A}\|_{\max}=1$ (or, more generally 
and for convenience, $\|\bar{A}\|_{\max}\leq 1$) and therefore, 
we will take $\sigma^2\Id$ as covariance matrix in our distributions. 

\section{Preliminary results}

We prove in this section some bounds on one-dimensional Gaussian
random variables as well as a proposition on the expectation of
positive random variables.  The main results of the paper will easily
follow from them.

\begin{proposition}\label{l4}
Let $\mu$, $\varsigma>0$ and $t>1$ be fixed numbers. Let $X\sim
N(\mu,\varsigma^2)$ be a normal distributed random variable. Then
$$
\Prob\{|X|>t|X+1|\}<
   \left(\frac{|\mu|+\varsigma}{\varsigma}\right)
   \left(\frac{1}{t-1}\right)\sqrt\frac{2}{\pi}.
$$
\end{proposition}

The proof of Proposition~\ref{l4} proceeds through a sequence of lemmas.

\begin{lemma}\label{l1}
Let $\mu\in\R$ and $\varsigma>0$ be fixed numbers. 
Let $X\sim  N(\mu,\varsigma^2)$ be a normal distributed random variable. 
Then
\begin{eqnarray*}
\Prob\{1<X<1+\varepsilon\}
&\leq&\frac{\varepsilon}{\varsigma}\sqrt\frac{1}{2\pi}.
\end{eqnarray*}
\end{lemma}

\begin{proof}
Since $X\sim N(\mu,\varsigma^2)$
\begin{eqnarray*}
\Prob\{1<X<1+\varepsilon\}
&=&\frac{1}{\varsigma}\sqrt\frac{1}{2\pi}\int_1^{1+\varepsilon}
e^{-\frac{(x-\mu)^2}{2\varsigma^2}}dx\\
&\leq&\frac{1}{\varsigma}\sqrt\frac{1}{2\pi}\int_1^{1+\varepsilon}
1\,dx 
\;=\;\frac{\varepsilon}{\varsigma}\sqrt\frac{1}{2\pi}.
\end{eqnarray*}
\end{proof}

\begin{lemma}\label{l2}
Let $\mu\in\R$, $\varsigma>0$. Let $X\sim
N(\mu,\varsigma^2)$ be a Gaussian random variable.
Then
$$
   \Prob\{1<X<1+\varepsilon\}\leq \varepsilon
   \left(\frac{|\mu|+\varsigma}{\varsigma}\right)\sqrt\frac{1}{2\pi}.
$$
\end{lemma}

\begin{proof}
We first assume that $\mu\geq 0$. 
Let $r=\frac{\varsigma}{\mu}$ and 
$f:\R\to\R$ be given by
\begin{eqnarray*}
f(m) &=& \frac{1}{mr}\sqrt\frac{1}{2\pi}\int_1^{1+\varepsilon}
e^{-\frac{(x-m)^2}{2m^2r^2}}dx
\end{eqnarray*}
so that $\Prob\{1<X<1+\varepsilon\}=f(\mu)$. By doing the change 
of variables $u=\frac{x-m}{mr\sqrt 2}$ we obtain,
\begin{eqnarray*}
f(m) &=& \sqrt\frac{1}{\pi}\int_{\frac{1-m}{mr\sqrt
2}}^\frac{1+\varepsilon-m}{mr\sqrt 2} e^{-u^2}du\\
&=&
\sqrt\frac{1}{\pi}\left(\int_0^\frac{1+\varepsilon-m}{mr\sqrt
2} e^{-u^2}du- \int_0^{\frac{1-m}{mr\sqrt
2}} e^{-u^2}du\right)
\end{eqnarray*}
and, hence,
$$
f'(m) = \sqrt\frac{1}{\pi}\left(\frac{d}{dm}\int_0^\frac{1+\varepsilon-m}
{mr\sqrt 2} e^{-u^2}du- \frac{d}{dm}\int_0^{\frac{1-m}{mr\sqrt 2}}
e^{-u^2}du\right).
$$
Let $v=\frac{1+\varepsilon-m}{mr\sqrt 2}$ and $w =
\frac{1-m}{mr\sqrt 2}$. Then
\begin{eqnarray*}
f'(m) &=& \sqrt\frac{1}{\pi}\left(\frac{d}{dm}\int_0^v e^{-u^2}du-
\frac{d}{dm}\int_0^w e^{-u^2}du\right).
\end{eqnarray*}
By the chain rule and the Fundamental Theorem of Calculus,
\begin{eqnarray}
f'(m) &=&
\sqrt\frac{1}{\pi}\left(\frac{dv}{dm}\cdot\frac{d}{dv}\int_0^v
e^{-u^2}du-\frac{dw}{dm}\cdot\frac{d}{dw}\int_0^w e^{-u^2}du\right)\notag\\
 &=&
\sqrt\frac{1}{\pi}\left(\frac{dv}{dm}e^{-v^2}-\frac{dw}{dm}e^{-w^2}\right).\label{eq52}
\end{eqnarray}
We now use that 
$$
\frac{dv}{dm}=\frac{-(1+\varepsilon)}{m^2r\sqrt
2}
\qquad\mbox{and}\qquad
\frac{dw}{dm}=\frac{-1}{m^2r\sqrt 2}
$$
to deduce from (\ref{eq52}) that 
\begin{eqnarray}
 -m^2r\sqrt{2\pi}f'(m)
&=&e^{-v^2}-(1+\varepsilon)e^{-w^2}\notag\\
&=&e^{-\frac{(1+\varepsilon-m)^2}{2m^2r^2}}-(1+\varepsilon)e^{-\frac{(1-m)^2}{2m^2r^2}}.\label{eq30}
\end{eqnarray}
Let $m_*$ be such that
$$
f(m_*)=\sup_{m\geq 0}f(m).
$$
Since $\displaystyle\lim_{m\rightarrow\infty}f(m)=
\displaystyle\lim_{m\rightarrow 0}f(m)=0$ we deduce that 
$f'(m_*)=0$. Equation (\ref{eq30}) evaluated at $m_*$ then yields
$$
e^{-\frac{(1-m_*)^2}{2m_*^2r^2}}
=(1+\varepsilon)e^{-\frac{(1+\varepsilon-m_*)^2}{2m_*^2r^2}}
$$
which elementary computations show equivalent to
\begin{eqnarray*}
\varepsilon^2+2\varepsilon(1-m_*)&=&2m_*^2r^2\ln(1+\varepsilon).
\end{eqnarray*}
Since $\ln(x)\leq x-1$ for all $x>0$ this last equality implies that
$$
\varepsilon+2-2m_*\leq 2m_*^2r^2
$$
which in turm implies, since $\varepsilon>0$,
\begin{equation*}
r^2 m_*^2+m_*-1>0.\label{eq31}
\end{equation*}
Solving this quadratic inequality we deduce that either
$$
  2r^2m_*>-1+\sqrt{1+4r^2}
$$
or
$$
  2r^2m_*<-1-\sqrt{1+4r^2}
$$
but we can reject the latter since $m_*\geq 0$. The former
inequality can also be written as
\begin{equation*}
m_*r >\frac{-1+\sqrt{1+4r^2}}{2r}
=\frac {4r^2}{2r(1+\sqrt{1+4r^2})}
=\frac {2r}{1+\sqrt{1+4r^2}}
\geq \frac {r}{1+r}.
\end{equation*}
Let $Y\sim N(m_*,m_*r)$. Using Lemma \ref{l1} and this
inequality we deduce that
$$
  f(m_*)\,=\,\Prob\{1<Y<1+\varepsilon\}\,\leq\,
  \frac{\varepsilon}{m_*r}\sqrt\frac{1}{2\pi}
  \,\leq\, \varepsilon\frac{1+r}{r}\sqrt\frac{1}{2\pi}
  \,=\, \varepsilon\frac{\mu+\varsigma}{\varsigma}\sqrt\frac{1}{2\pi}.
$$
The statement (for the case $\mu\geq 0$) now follows since
$$
  \Prob\{1<X<1+\varepsilon\}\,=\, f(\mu)\,\leq\,
   f(m_*).
$$

We next deal with the case $\mu<0$. 
Since $X\sim N(\mu,\varsigma^2)$,
\begin{eqnarray}
\Prob\{1<X<1+\varepsilon\}
&=&\frac{1}{\varsigma}\sqrt\frac{1}{2\pi}\int_1^{1+\varepsilon}
e^{-\frac{(x-\mu)^2}{2\varsigma^2}}dx.\label{eq36}
\end{eqnarray}
Let $Y\sim N(-\mu,\varsigma^2)$. Then
\begin{eqnarray}
\Prob\{1<Y<1+\varepsilon\}
&=&\frac{1}{\varsigma}\sqrt\frac{1}{2\pi}\int_1^{1+\varepsilon}
e^{-\frac{(x+\mu)^2}{2\varsigma^2}}dx.\label{eq37}
\end{eqnarray}
Since $(x+\mu)^2<(x-\mu)^2$ for all $x\in(1,1+\varepsilon)$,
using (\ref{eq36}) and (\ref{eq37}) we obtain
$$
\Prob\{1<Y<1+\varepsilon\} \geq\Prob\{1<X<1+\varepsilon\}.
$$
The result now follows since, by the first case above, 
the claimed bound holds for $Y$. 
\end{proof}

\proofof{Proposition~\ref{l4}}
We have
\begin{eqnarray*}
|X|>t|X+1|&\iff& X^2>t^2(X+1)^2\\
&\iff&(t^2-1)X^2+2t^2X+t^2<0\\
&\iff&\frac{- t}{t-1}<X<\frac{-t}{t+1}\\
&\iff&\frac{t+1}{t-1}>\left(-\frac{t+1}{t}\right)X>1\\
&\iff&1+\frac{2 }{t- 1}>\left(-\frac{t+1}{t}\right)X>1.
\end{eqnarray*}
Letting $Y=\left(-\frac{t+1}{t}\right)X$ we conclude that
\begin{eqnarray}
\Prob\{|X|>t|X+1|\}
 &=&\Prob\left\{1<Y<1+\frac{2}{t- 1}\right\}.\label{eq20}
\end{eqnarray}
Since $Y=\left(-\frac{t+1}{t}\right)X$, $Y\sim
N(\mu_Y,\varsigma_Y^2)$ where
\begin{eqnarray}
\mu_Y =\left(-\frac{t+1}{t}\right)\mu\qquad\mbox{and}\qquad
\varsigma_Y^2 =\left(-\frac{t+1}{t}\right)^2\varsigma^2.\label{eq71}
\end{eqnarray}
We now apply Lemma~\ref{l2} to $Y$ with $\varepsilon=\frac{2}{t-1}$ to
obtain
\begin{eqnarray}
   \Prob\left\{1<Y<1+\frac{2}{t-1}\right\}&\leq&
   \left(\frac{\mu_Y+\varsigma_Y}{\varsigma_Y}\right)
   \left(\frac{1}{t-1}\right)\sqrt\frac{2}{\pi}.\label{eq21}
\end{eqnarray}
Combining (\ref{eq20}), (\ref{eq71}) and (\ref{eq21}) the proof is
done. \eproof

The following proposition is a variation of a classical result
for positive random variables (cf.~\cite[Proposition~2]{ChC08}).

\begin{proposition}\label{prop1}
Let $k,H>0$ and $X>1$ be a random variable satisfying
$\Prob\{X>t\}\leq \frac{k}{t-H}$ for all $t>k+H$. Then, 
for all $\beta>1$, 
$$
   \bE(\log_\beta(X))<\log_\beta\left(k+H\right)+\frac{1}{\ln\beta}.
$$
\end{proposition}

\begin{proof}
We have
\begin{eqnarray*}
\bE(\log_\beta(X))
&=& \int_0^\infty \Prob\{\log_\beta(X)>s\}ds
\,=\,\int_0^\infty \Prob\{X>\beta^s\}ds\\
&=&\int_0^{\log_\beta\left(k+H\right)}
\Prob\{X>\beta^s\}ds+\int_{\log_\beta\left(k+H\right)}^\infty
\Prob\{X>\beta^s\}ds\\
&\leq&\log_\beta\left(k+H\right)+\int_{\log_\beta\left(k+H\right)}^\infty
\Prob\{X>\beta^s\}ds.
\end{eqnarray*}
Since $\Prob\{X>t\}\leq \frac{k}{t-H}$ it follows that
\begin{equation}
\bE(\log_\beta(X))-\log_\beta\left(k+H\right) \leq 
k\int_{\log_\beta\left(k+H\right)}^\infty
\frac{dt}{\beta^t-H}.\label{eq64}
\end{equation}
Let $u=\frac{H}{\beta^t-H}$ so that $du=-\ln\beta(u+u^2)dt$. 
Then, changing variables in~(\ref{eq64}), we obtain
\begin{equation*}
\bE(\log_\beta(X))-\log_\beta\left(k+H\right)
\,\leq\,-\frac{k}{H\ln\beta}\int_{\frac{H}{k}}^0\frac{du}{1+u}
\,=\,\frac{k}{H\ln\beta}\ln\left(1+\frac{H}{k}\right).
\end{equation*}
The proof is complete since $\ln(1+x)<x$ for all $x>-1$.
\end{proof}

\section{Computation of determinants}

In this section we consider the problem of computing the determinant. 
Taking $F(A)=\det(A)$ in~\eqref{def:c} we obtain the componentwise 
condition number $\cdet(A)$ for this problem. Our main result for this 
quantity is the following.

\begin{theorem}\label{thm:det}
Let $S\subset[n]^2$ be admissible, $\bar{A}\in\MM_S$ 
with $\|\bar{A}\|_{\max}\leq 1$, $\sigma>0$  and 
$A\sim N_S(\bar{A},\sigma^2\Id)$. Then, 
for any real number $t>|S|$,
$$
 \Prob\{\cdet(A)>t\}<\left(\frac{1+\sigma}{\sigma}\right)
 \left(\frac{|S|^2}{t-|S|}\right)\sqrt\frac{2}{\pi}
$$
and, for all $\beta>1$,
$$
  \bE(\log_\beta(c_{\det}(A)))<\log_\beta
  \left(\frac{1+\sigma}{\sigma}\right)+2\log_\beta|S|
  +\frac{1.03}{\ln\beta}.
$$
\end{theorem}

For the proof of this theorem we will make use of
the following characterization of $\cdet(A)$ 
(see~\cite[Lemma~1.1]{ChC08} for a proof).  
Denote by $\gamma_{ij}$ the entry of $A^{-1}$ on the $i$th row and
$j$th column. Then, for any matrix $A\in\MM\setminus\Sigma$, 
\begin{equation}
  \cdet(A)=\sum_{i,j\in[n]}\left|a_{ij}\gamma_{ji}\right|.\label{old_result}
\end{equation}

\proofof{Theorem~\ref{thm:det}}
Without loss of generality, we may assume 
that $(1,1)\in S$ so that $a_{11}\sim N(\bar{a}_{11},\sigma^2)$. 
For a time to come we consider all entries of $A$ except
$a_{11}$ to be fixed. Let $A_{ij}$ be the matrix obtained by
removing from $A$ the $i$th row and $j$th column. By Cramer's
rule, $\gamma_{11}=\frac{\det(A_{11})}{\det(A)}$ and therefore,
for $t>1$,
$$
  \Prob\{\left|a_{11}\gamma_{11}\right|>t\} =
 \Prob\big\{\left|a_{11}\det(A_{11})\right|>t|\det(A)|\big\}.
$$
Expanding $\det(A)$ by the first column of $A$ this 
equality becomes
\begin{equation*}
\Prob\{\left|a_{11}\gamma_{11}\right|>t\}
=\Prob\left\{\left|a_{11}\det(A_{11})\right|>t\left|\sum_{i=1}^n
(-1)^{i+1}a_{i1}\det(A_{i1})\right|\right\} 
\end{equation*}
and letting
\begin{equation*}
X := \frac{a_{11}\det(A_{11})}{\sum_{i=2}^n
(-1)^{i+1}a_{i1}\det(A_{i1})}.
\end{equation*}
this equality becomes
\begin{eqnarray}
\Prob\{\left|a_{11}\gamma_{11}\right|>t\}
&=&\Prob\left\{\left|X\right|>t\left|X+1\right|\right\} .\label{eq43}
\end{eqnarray}
Since all entries of $A$, except $a_{11}$ are fixed (and
$a_{11}\sim N(\bar{a}_{11},\sigma^2)$), we have 
$X\sim N(\mu,\varsigma^2)$, where
$$
\mu=\frac{\bar{a}_{11}\det(A_{11})}{\sum_{i=2}^n
(-1)^{i+1}a_{i1}\det(A_{i1})}\quad \mbox{and}\quad
\varsigma=\left|\frac{\sigma\det(A_{11})}{\sum_{i=2}^n
(-1)^{i+1}a_{i1}\det(A_{i1})}\right|.
$$
In particular,
\begin{equation}\label{eq45}
\frac{|\mu|+\varsigma}{\varsigma}\,=\,
\frac{|\bar{a}_{11}|+\sigma}{\sigma}\,\leq\,
\frac{1+\sigma}{\sigma}
\end{equation}
the last since $\|\bar{A}\|_{\max}\leq 1$. 
By Proposition~\ref{l4}, and Equations~(\ref{eq43}) and~(\ref{eq45}), we have
\begin{eqnarray*}
\Prob\{\left|a_{11}\gamma_{11}\right|>t\}
&\leq&\left(\frac{1+\sigma}{\sigma}\right)
\left(\frac{1}{t-1}\right)\sqrt\frac{2}{\pi}.
\end{eqnarray*}
This inequality holds for all fixed
values of $a_{12}, a_{13}, ...a_{nn}$. Therefore, it holds as well
when all entries of $A$ are random (as described in
Section~\ref{s2}). 
We can show in the same manner that, for all $(i,j)\in S$,
\begin{eqnarray}
\Prob\{\left|a_{ij}\gamma_{ji}\right|>t\}
&\leq&\left(\frac{1+\sigma}{\sigma}\right)
\left(\frac{1}{t-1}\right)\sqrt\frac{2}{\pi}.\label{eq51}
\end{eqnarray}
We now recall that, for all $(i,j)\not\in S$, $a_{ij}=0$. 
Hence, by using (\ref{old_result}), for $t>|S|$,
\begin{eqnarray}
\Prob\{\cdet(A)>t\} 
&=&\Prob\bigg\{\sum_{(i,j)\in [n]^2}\left|a_{ij}\gamma_{ji}\right|>t\bigg\}\notag\\
&=&\Prob\bigg\{\sum_{(i,j)\in S}
  \left|a_{ij}\gamma_{ji}\right|>t\bigg\}\notag\\
&\leq&\sum_{(i,j)\in S}
  \Prob\bigg\{\left|a_{ij}\gamma_{ji}\right|>\frac{t}{|S|}\bigg\}\notag\\
&\leq&\sum_{(i,j)\in  S}\left(\frac{1+\sigma}{\sigma}\right)
\left(\frac{|S|}{t-|S|}\right)\sqrt\frac{2}{\pi}
  \quad\mbox{[by (\ref{eq51})]}\notag\\
&=&\left(\frac{1+\sigma}{\sigma}\right)
\left(\frac{|S|^2}{t-|S|}\right)\sqrt\frac{2}{\pi}.\label{eq62}
\end{eqnarray}

Combining Equation (\ref{eq62}) and Proposition~\ref{prop1} 
we obtain
\begin{align*}
\bE(\log_\beta&\cdet(A))\\
 &\leq\;
\log_\beta\left(|S|+\left(\frac{1+\sigma}{\sigma}\right)
|S|^2\sqrt\frac{2}{\pi}\right)+\frac{1}{\ln\beta}\\
&=\;\log_\beta\left(\left(\frac{1+\sigma}{\sigma}\right)
|S|^2\sqrt\frac{2}{\pi}\left(1+\left(\frac{\sigma}{1+\sigma}\right)
\frac{1}{|S|}\sqrt\frac{\pi}{2}\right)\right)+\frac{1}{\ln\beta}\\
&=\;\log_\beta\left(\left(\frac{1+\sigma}{\sigma}\right)
|S|^2\sqrt\frac{2}{\pi}\right)+\log_\beta\left(1+\left(\frac{\sigma}{1+\sigma}\right)
\frac{1}{|S|}\sqrt\frac{\pi}{2}\right)+\frac{1}{\ln\beta}\\
&\leq\;\log_\beta\left(\left(\frac{1+\sigma}{\sigma}\right)
|S|^2\sqrt\frac{2}{\pi}\right)+\frac{1}{\ln\beta}
\left(\frac{\sigma}{1+\sigma}\right)
\frac{1}{|S|}\sqrt\frac{\pi}{2}+\frac{1}{\ln\beta}.
\end{align*}
The last line above is true because $\log_\beta(1+x)\leq \frac{x}{\ln\beta}$ 
for all $x\geq 0$.
Since both $\sigma$ and $|S|>0$,
\begin{align}
\bE(\log_\beta\cdet(A)) \;\leq\;&
\log_\beta\left(\left(\frac{1+\sigma}{\sigma}\right)
|S|^2\sqrt\frac{2}{\pi}\right)+ \frac{1}{\ln\beta}\left(\sqrt\frac{\pi}{2}+1\right)
\notag\\
\leq\;& \log_\beta \left(\frac{1+\sigma}{\sigma}\right)
+2\log_\beta|S| +\frac{1.03}{\ln\beta}.\tag*{\qed}
\end{align}

\section{Matrix inversion}

We next consider the problem of matrix inversion. 
For $k,l\in [n]$ we consider the 
function $F_{kl}:\MM_S\setminus\Sigma\to\MM$ given 
by $F_{kl}(A)=(A^{-1})_{kl}$. Definition~\eqref{def:c} 
applied to this function yields a componentwise condition 
number $\ckl(A)$ and, recall~\eqref{char:c}, taking the 
maximum over $(k,l)\in[n]^2$ we obtain 
$\ci(A)$. Our main result for this quantity is the following. 

\begin{theorem}\label{result2} 
Let $S\subset[n]^2$ be admissible, $\bar{A}\in\MM_S$ 
such that $\|\bar{A}\|_{\max}\leq 1$, 
$\sigma>0$  and 
$A\sim N_S(\bar{A},\sigma^2\Id)$. Then, 
for any real number $t>2|S|$,
\begin{eqnarray*}
\Prob\{\ci(A)>t\} &=&\left(\frac{1+\sigma}{\sigma}\right)
\left(\frac{4n^2|S|^2}{t-2|S|}\right)\sqrt\frac{2}{\pi}.
\end{eqnarray*}
and, for all $\beta>1$, 
\begin{eqnarray*}
\bE(\log_\beta(\ci(A)))
&=&\log_\beta\left(\frac{1+\sigma}{\sigma}\right)
+2\log_\beta(n|S|)+ \frac{2.65}{\ln\beta}.
\end{eqnarray*}
\end{theorem}

\begin{lemma}(\cite[Lemma~5]{ChC08})\label{l6}
For $A\in\MM\setminus\Sigma$ and $k,l\in[n]$,
\begin{equation}\tag*{\qed}
   \ckl(A) \leq \cdet(A) + \cdet(A_{lk}).
\end{equation}
\end{lemma}
\vspace{0.3cm} 

\proofof{Theorem~\ref{result2}} 
Almost certainly, $A\in\MM\setminus\Sigma$. Hence, 
by Lemma~\ref{l6}, we have, for all $k,l\in[n]$,
\begin{eqnarray*}
\Prob\{\ckl(A)>t\} &\leq& \Prob\{\cdet(A) +
\cdet(A_{lk})>t\}\\
&\leq& \Prob\left\{\cdet(A)>\frac{t}{2}\mbox{ or }
\cdet(A_{lk})>\frac{t}{2}\right\}\\
&\leq& \Prob\left\{\cdet(A)>\frac{t}{2}\right\}+\Prob\left\{
\cdet(A_{lk})>\frac{t}{2}\right\}.
\end{eqnarray*}
Using Theorem~\ref{thm:det} twice, we obtain
\begin{eqnarray*}
\Prob\{\ckl(A)>t\} &\leq&\left(\frac{1+\sigma}{\sigma}\right)
\left(\frac{|S|^2}{\frac{t}{2}-|S|}\right)\sqrt\frac{2}{\pi}
+\left(\frac{1+\sigma}{\sigma}\right)
\left(\frac{|S|^2}{\frac{t}{2}-|S|}\right)\sqrt\frac{2}{\pi}\notag\\
&=&\left(\frac{1+\sigma}{\sigma}\right)
\left(\frac{4|S|^2}{t-2|S|}\right)\sqrt\frac{2}{\pi}.
\end{eqnarray*}
This inequality and the definition of $\ci(A)$ yield 
\begin{eqnarray*}
\Prob\{\ci(A)>t\} &=&\Prob\left\{\max_{k,l\in[n]}\ckl(A)>t\right\}\\
&\leq&\sum_{k,l\in[n]}\Prob\left\{\ckl(A)>t\right\}\\
&\leq&\sum_{k,l\in[n]}\left(\frac{1+\sigma}{\sigma}\right)
\left(\frac{4|S|^2}{t-2|S|}\right)\sqrt\frac{2}{\pi}\\
&=&\left(\frac{1+\sigma}{\sigma}\right)
\left(\frac{4n^2|S|^2}{t-2|S|}\right)\sqrt\frac{2}{\pi}.
\end{eqnarray*}
Finally, by Proposition~\ref{prop1}

\begin{align*}
\bE(\log_\beta&(\ci(A)))\\
&\leq\; \log_\beta\left(2|S|+\left(\frac{1+\sigma}{\sigma}\right)
    \left(4n^2|S|^2\right)\sqrt\frac{2}{\pi}\right)+\frac{1}{\ln\beta}\\
&\leq\; \log_\beta\left(\left(\frac{1+\sigma}{\sigma}\right)
  \left(4n^2|S|^2\right)\sqrt\frac{2}{\pi}\left(1+\sqrt\frac{\pi}{8}\right)\right)
  +\frac{1}{\ln\beta}\\
&=\; \log_\beta\left(\left(\frac{1+\sigma}{\sigma}\right)
  \left(n^2|S|^2\right)\right)+
  \log_\beta\left(\sqrt{\frac{32}{\pi}}\left(1+\sqrt\frac{\pi}{8}\right)\right)
 +\frac{1}{\ln\beta}\\
&\leq\; \log_\beta\left(\left(\frac{1+\sigma}{\sigma}\right)
  \left(n^2|S|^2\right)\right)+\frac{2.65}{\ln\beta},
\end{align*}
the second inequality due to the fact that $n,|S|\geq 1$ and $\sigma>0$. 
\eproof

\section{Linear equations solving}

We finally consider linear equation solving. 
For $A\in\MM\setminus\Sigma$ and $b\in\R^n$ we compute $x=A^{-1}b$. 
Thus, for $k\in[n]$, the mapping $(A,b)\mapsto x_k$ yields 
(always using~\eqref{def:c}) $\ck(A,b)$ and 
taking the maximum over $k\in[n]$ we obtain the componentwise 
condition number $\c(A,b)$ of the pair $(A,b)$. The following 
theorem is the main result in this section.

\begin{theorem}\label{result3} 
Let $S\subset[n]^2$ be admissible, $\bar{A}\in\MM_S$ and 
$\bar{b}\in\R^n$ such that $\|\bar{A}\|_{\max}\leq 1$ and 
$\|\bar{b}\|_\infty\leq 1$, $\sigma>0$, 
$A\sim N(\bar{A},\sigma^2\Id)$ and $b\sim N(\bar{b},\sigma^2\Id)$. 
Then, for any real number $t>2|S|$,
\begin{eqnarray*}
\Prob\{\c(A,b)>t\} &=&\left(\frac{1+\sigma}{\sigma}\right)
\left(\frac{4n|S|^2}{t-2|S|}\right)\sqrt\frac{2}{\pi}.
\end{eqnarray*}
and, for all $\beta>1$, 
\begin{eqnarray*}
\bE(\log_\beta(\ci(A)))
&=&\log_\beta\left(\frac{1+\sigma}{\sigma}\right)+2\ln |S| 
+\log_\beta n+\frac{2.65}{\ln\beta}.
\end{eqnarray*}
\end{theorem}

In what follows let $R_k$ be the matrix obtained by replacing the
$k$th column of $A$ by $b$. 

\begin{lemma}(\cite[Lemma~6]{ChC08})\label{l7}
For any non-singular matrix $A$ and $k\in[n]$,
\begin{equation}\tag*{\qed}
 \ck(A,b)\leq \cdet(A)+\cdet(R_k).
\end{equation}
\end{lemma}
\vspace{0.3cm} 

\proofof {Theorem~\ref{result3}} 
By Lemma~\ref{l7}, we have, for all $k\in[n]$,
\begin{eqnarray*}
\Prob\{\ck(A,b)>t\}&\leq& \Prob\{\cdet(A) +
\cdet(R_k)>t\}\\
&\leq& \Prob\left\{\cdet(A)>\frac{t}{2}\mbox{ or }
\cdet(R_k)>\frac{t}{2}\right\}\\
&\leq& \Prob\left\{\cdet(A)>\frac{t}{2}\right\}+\Prob\left\{
\cdet(R_k)>\frac{t}{2}\right\}.
\end{eqnarray*}
It follows from our hypothesis that $\|R_k\|_{\max}\leq 1$.
We can therefore apply Theorem~\ref{thm:det} twice 
to obtain 
\begin{equation*}
\Prob\{\ck(A,b)>t\} \leq \left(\frac{1+\sigma}{\sigma}\right)
\left(\frac{4|S|^2}{t-2|S|}\right)\sqrt\frac{2}{\pi}
\end{equation*}
and, proceeding as in the proof of Theorem~\ref{result2},   
\begin{eqnarray*}
\Prob\{\c(A,b)>t\}&=&\Prob\left\{\max_{k\in[n]}\ck(A,b)>t\right\}\\
&\leq&\sum_{k\in[n]}\Prob\left\{\ck(A,b)>t\right\}\\
&\leq&\sum_{k\in[n]}\left(\frac{1+\sigma}{\sigma}\right)
\left(\frac{4|S|^2}{t-2|S|}\right)\sqrt\frac{2}{\pi}\\
&=&\left(\frac{1+\sigma}{\sigma}\right)
\left(\frac{4n|S|^2}{t-2|S|}\right)\sqrt\frac{2}{\pi}.
\end{eqnarray*}
A last call to Proposition~\ref{prop1} yields
the desired bound for $\bE(\log_\beta(\c(A,b))$.
\eproof

\section{On the accuracy of forward substitution}

We arrive, at last, to the motivating theme of this paper. 
Theorem~\ref{thm:LES} is an immediate consequence 
of Theorem~\ref{result3} since lower triangular matrices 
are sparse matrices with $S=\{(i,j)\in[n]^2\mid i\geq j\}$.
One then only needs to use that $|S|=\frac{n(n+1)}{2}$. 

For the proof of Corollary~\ref{cor:LES} we use a common 
approach, pioneered by Wilkinson, which splits the relative error 
bound in the computed solution $\RelError(F(a))$ 
as the product of two factors, 
one depending on the algorithm but not on the data (a
backward error bound) and another depending on the data but 
not on the algorithm used (the condition of the data). A backward 
error bound for forward substitution is shown in the following 
result, going back to Wilkinson~\cite[Ch.3,\S19]{Wilkinson63}, 
which we quote, omitting some smaller details, 
as given in~\cite[Proposition~3.5]{Condition}.

\begin{proposition}\label{prop:back}
Let $L=(l_{ij})\in\R^{n\times n}$ be a nonsingular triangular matrix, 
$b\in\R^n$, and assume 
$\emac$ is sufficiently small (of the order of $(\log n)^{-1}$). Then, 
the solution $\hat x$ of the system $Lx=b$ computed with forward substitution 
satisfies 
$$
   (L+E)\hat x=b,
$$
where
\begin{equation}\tag*{\qed}
   \frac{|e_{ij}|}{|l_{ij}|}\leq (2\log_2 n)\emac. 
\end{equation}
\end{proposition}

Proposition~\ref{prop:back} yields a 
backward error bound of the form $B\emac$ 
where $B=2\log_2 n$ is an expression in the dimension $n$
of the input, independent of $\emac$. 

The way such a backward error bound combines 
with condition to produce a bound for the loss of precision, in 
digits, is (see Theorem~O.3 in~\cite{Condition})
$$
   \LoP(F(a)) \leq \log_{10} B +\log_{10} \cond^F(a) +o(1).
$$ 
Here the $o(1)$ term is an expression tending to zero as $\emac$ does so, 
$\cond^F(a)$ is the condition number of $a$ and ---crucially in our context--- 
if the bound $B\emac$ is componentwise, as in Proposition~\ref{prop:back}, 
this condition number  can be taken componentwise as well. Doing so for 
forward substitution and $x=L^{-1}b$ we obtain 
$$
   \LoP(x) \leq \log_{10} (2\log_2 n) +\log_{10} \c(L,b) +o(1).
$$
Taking expectations on both sides and using Theorem~\ref{thm:LES} 
proves Corollary~\ref{cor:LES}.

\bibliographystyle{plain}

\begin{thebibliography}{10}

\bibitem{Condition}
P.~B\"urgisser and F.~Cucker.
\newblock {\em Condition}.
\newblock Forthcoming in {\em Grundleheren der mathematischen Wissenschaften},
  Springer-Verlag.

\bibitem{ChC08}
D.~Cheung and F.~Cucker.
\newblock Componentwise condition numbers of random sparse matrices.
\newblock {\em SIAM J. Matrix Anal. Appl.}, 31:721--731, 2009.

\bibitem{Higham:89}
N.~Higham.
\newblock The accuracy of solutions to triangular systems.
\newblock {\em SIAM J. Numer. Anal.}, 26:1252--1265, 1989.

\bibitem{ST:02}
D.A. Spielman and S.-H. Teng.
\newblock Smoothed analysis of algorithms.
\newblock In {\em Proceedings of the International Congress of Mathematicians},
  volume~I, pages 597--606, 2002.

\bibitem{ST:04}
D.A. Spielman and S.-H. Teng.
\newblock Smoothed analysis: Why the simplex algorithm usually takes polynomial
  time.
\newblock {\em Journal of the ACM}, 51(3):385--463, 2004.

\bibitem{ST:06}
D.A. Spielman and S.-H. Teng.
\newblock Smoothed analysis of algorithms and heuristics.
\newblock In {\em Foundations of Computational Mathematics, Santander 2005},
  volume 331 of {\em Lecture Notes of the London Mathematical Society}, pages
  274--342, 2006.

\bibitem{ST:09}
D.A. Spielman and S.-H. Teng.
\newblock Smoothed analysis: An attempt to explain the behavior of algorithms
  in practice.
\newblock {\em Communications of the ACM}, 52(10):77--84, 2009.

\bibitem{Turing48}
A.M. Turing.
\newblock Rounding-off errors in matrix processes.
\newblock {\em Quart. J. Mech. Appl. Math.}, 1:287--308, 1948.

\bibitem{ViTr98}
D.~Viswanatah and L.N. Trefethen.
\newblock Condition numbers of random triangular matrices.
\newblock {\em SIAM J. Matrix Anal. Appl.}, 19:564--581, 1998.

\bibitem{vNGo47}
J.~von Neumann and H.H. Goldstine.
\newblock Numerical inverting matrices of high order.
\newblock {\em \BAMS}, 53:1021--1099, 1947.

\bibitem{Wilkinson63}
J.~Wilkinson.
\newblock {\em Rounding Errors in Algebraic Processes}.
\newblock Prentice Hall, 1963.

\end{thebibliography}


\end{document}